\documentclass[12pt]{amsart}
\usepackage{amssymb,amsmath,amstext}
\usepackage{graphicx}

\theoremstyle{plain}

\newtheorem{theorem}{Theorem}[section]

\newtheorem{proposition}[theorem]{Proposition}
\newtheorem{lemma}[theorem]{Lemma}

\newtheorem{corollary}[theorem]{Corollary}
\newtheorem{claim}[theorem]{Claim}
\newtheorem{conjecture}[theorem]{Conjecture}
 
\theoremstyle{definition}
\newtheorem{definition}[theorem]{Definition}

\newtheorem{problem}[theorem]{Problem}

\theoremstyle{remark}
\newtheorem{example}[theorem]{Example}

\begin{document}

\newcommand{\Dkk}{\Delta_k \times \Delta_k}
\newcommand{\aij}{A_{ij}}
\newcommand{\cf}{\mathcal{F}}
\newcommand{\F}{\mathcal{F}}
\newcommand{\I}{\mathcal{I}}
\newcommand{\J}{\mathcal{J}}
\newcommand{\M}{\mathcal{M}}

\newcommand{\T}{\mathcal{T}}

\newcommand{\e}{\varepsilon}
\newcommand{\vj}{\vec{j}}
\newcommand{\vw}{\vec{w}}
\newcommand{\vz}{\vec{z}}
\newcommand{\vh}{\vec{h}}
\newcommand{\vx}{\vec{x}}
\newcommand{\vy}{\vec{y}}
\newcommand{\vk}{\vec{k}}
\newcommand{\vv}{\vec{v}}
\newcommand{\vl}{\vec{\ell}}
\newcommand{\cl}{\mathcal L}
\newcommand{\cm}{\mathcal M}
\newcommand{\ch}{\mathcal H}
\newcommand{\ci}{\mathcal I}

\newcommand{\cn}{\mathcal N}
\newcommand{\dist}{\text{dist}}
\newcommand{\tu}{\tau^{(2)}}
\newcommand{\tuh}{\tau^{(2)}(H)}
\newcommand{\tm}{\tau^{(m)}}
\newcommand{\tmh}{\tau^{(m)}(H)}
\newcommand{\tst}{{{\tau}^*}^{(2)}}
\newcommand{\tsm}{{\tau^*}^{(m)}}
\newcommand{\chkm}{\ch(k,m)}
\newcommand{\chtt}{\ch(3,2)}

\newcommand{\chkmi}{\ch(k,m)^i}
\newcommand{\nt}{\nu^{(2)}}
\newcommand{\nth}{\nu^{(2)}(H)}
\newcommand{\nm}{\nu^{(m)}}
\newcommand{\nmh}{\nu^{(m)}(H)}
\newcommand{\nsm}{{\nu^*}^{(m)}}

\newcommand{\nst}{{\nu^*}^{(2)}}

\newcommand{\hu}{H^{(2)}}
\newcommand{\hm}{H^{(m)}}
\newcommand{\jm}{J^{(m)}}
\newcommand{\bkm}{\binom{k}{m}}


\title[A generalization of Tuza's conjecture]{A generalization of Tuza's conjecture}

\author{Ron Aharoni}\thanks{The research of the first  author was
supported by  BSF grant no.
2006099, by an ISF grant and by the Discount Bank
Chair at the Technion.}
\address{Department of Mathematics,
Technion\\
Haifa, Israel} \email{raharoni@gmail.com}

\author{Shira Zerbib}
\address{Department of Mathematics,
University of Michigan, Ann Arbor} \email{zerbib@umich.edu}


\begin{abstract}
A famous conjecture of Tuza \cite{tuza} is that the minimal number of edges needed to cover all triangles in a graph is at most twice the maximal size  of a set of edge-disjoint triangles. We propose a wider setting for this conjecture.
 For a hypergraph $H$ let $\nmh$ be the maximal size of a collection of edges,  no two of which share $m$ or more vertices,
and let $\tmh$ be the minimal size of a collection $C$ of sets of $m$ vertices, such that every edge in $H$ contains a set from $C$. We conjecture that the maximal ratio $\tmh/\nmh$ is attained in hypergraphs for which  $\nmh=1$. This would imply, in particular, the following generalization of Tuza's conjecture: if $H$ is $3$-uniform, then $\tuh/\nth \le 2$. (Tuza's conjecture is the case in which $H$ is the set of all triples of vertices of triangles in any given graph.)
We show that most known results on Tuza's conjecture go over to this more general setting. We also prove some general results on the ratio $\tmh/\nmh$, and study the fractional versions and the case of $k$-partite hypergraphs.
\end{abstract}

\maketitle


\section{Introduction}
\subsection{Terminology}

A {\em hypergraph} is a pair $H=(V=V(H),E=E(H))$, where $E$ is a collection of subsets of $V$, called {\em edges}. The elements of $V$ are called {\em vertices}. If all edges are of the same size $k$ then $H$ is said to be $k$-{\em uniform}. A hypergraph is called {\em linear}
if no two edges intersect in more than one vertex.

We shall often identify the hypergraph with its edge set, and write $H$ for $E(H)$. The {\em degree} $\deg_H(v)$ of a vertex $v$ in a hypergraph $H$ is the number of edges  containing $v$.  Let   $\Delta(H)=\max_{v \in V(H)}deg_H(v)$.
We shall  use the abbreviation $a_1a_2\dots a_m$ to denote the set  $\{a_1,a_2,\dots ,a_m\}$, and in this notation we assume that all $a_i$ are distinct. The notation $\binom{S}{m}$ stands for the set of all subsets of size $m$ of the set $S$.

A set of vertices in a hypergraph $H$ is called {\em independent} if it does not contain an edge. We denote by $\alpha(H)$  the maximal size of an independent set in $H$ (this notation is used in the paper only for graphs). The {\em line graph} $L(H)$ of a hypergraph $H$ has $E(H)$ as its vertices, and two vertices are connected if the corresponding edges intersect.

A {\em matching} in a hypergraph $H$ is a set of disjoint edges. 
The {\em matching number} $\nu(H)$  is the maximal size of a matching in $H$. Clearly $\nu(H)=\alpha(L(H)$. A {\em cover} is a set of vertices intersecting all edges of $H$. The {\em covering number} $\tau(H)$  is the minimal size of a cover.

\subsection{Tuza's conjecture}
Trivially, $\nu(H) \le \tau(H)$ for any hypergraph $H$. The union of all edges in a maximal matching is a cover, and hence if $H$ is $k$-uniform then 
\begin{equation}\label{knu}
    \tau(H) \le k\nu(H).
\end{equation}
Equality is attained, for example, in $k$-uniform projective planes, where $\nu=1$ and $\tau=k$, and in the complete $k$-uniform hypergraph $\binom{[(n+1)k-1]}{k}$, namely, the collection of all $k$-subsets of a set of size $(n+1)k-1$, where $\nu=n$ and $\tau=kn$. However, in special classes of hypergraphs, the inequality can be improved. For example, it is conjectured  that in $k$-partite hypergraphs $\tau\le (k-1)\nu$ (this conjecture is commonly attributed to Ryser, and it appeared in a thesis of his student, Henderson \cite{henderson}). Ryser's conjecture is known for $k=3$ \cite{aharoni}, so  in $3$-uniform $3$-partite hypergraphs $\tau\le 2\nu$. A famous conjecture of Tuza is that the same inequality, $\tau \le 2 \nu$, holds for the hypergraph whose vertices are the edges of a graph $G$, and its edges are the triangles in $G$, for any graph $G$. Below (Problem \ref{daring}) we shall offer some indication that the similarity between Tuza's conjecture and the case $k=3$ of Ryser's conjecture is not merely formal.

The main purpose of this article is to  put Tuza's conjecture in a general context. We shall formulate a more general conjecture (Conjecture \ref{general32}), that applies to all $3$-uniform hypergraphs, and then extend the latter to general uniformities (Conjecture \ref{ceiltuza}).
There are no new results in the paper on Tuza's conjecture itself, but the discussion of the general case and of higher uniformities may shed light on the difficulties it poses.  

\subsection{More notation}
We say that a hypergraph is an $m$-{\em matching} if no two edges in it intersect in $m$ vertices or more. So, a $1$-matching is plainly a matching.
For a hypergraph $H$ and a positive integer $m$ let $H^{(m)}$ be the hypergraph whose vertex set is $\binom{V(H)}{m}$ and whose edge set is $\{{e \choose m} \mid e \in E(H)\}$.
Clearly, a subset $F$ of $H$ is an $m$-matching if and only if $\{\binom{f}{m} \mid f \in F\}$ is a matching in $\hm$.

A cover of $\hm$ is said to be an $m$-{\em cover} of $H$. So, an $m$-cover of $H$ is a set $C$ of $m$-sets of vertices, such that every edge in $H$ contains at least one member of $C$. We write   $\nm(H)=\nu(\hm)$ and $\tm(H)=\tau(\hm)$, namely, $\nm(H)$ and $\tm(H)$ are the maximal size of an $m$-matching and the minimal size  of an $m$-cover in $H$, respectively.  

For $m<k$ let
$\ch(k,m)=\{\hm \mid H ~\text{is a } k\text{-uniform hypergraph}\}$. Define
 $$g(k,m)=\max\{\tau(J) \mid J \in \ch(k,m), ~\nu(J)=1\}$$ and
$$h(k,m)=\sup\Big\{\frac{\tau(J)}{\nu(J)} \mid J \in \ch(k,m)\Big\}.$$
Clearly, $g(k,m)\le h(k,m)$. 

We shall also consider the fractional versions of these parameters. Let  $H=(V,E)$ be a hypergraph. 
A function $f:E \to \mathbb{R}_{\ge 0}$ is a \emph{fractional matching} in $H$ if $\sum_{e \ni v} f(e) \le 1$ for every $v \in V$. The \emph{fractional matching number} $\nu^*(H)$ is the maximum of $\sum_{e \in E} f(e)$ over all fractional matchings $f$. 
A  function $g:V \to \mathbb{R}_{\ge 0}$ is a \emph{fractional cover} in $H$ if $\sum_{v \in e} g(v) \ge 1$ for every $e \in E$.  The \emph{fractional covering number} $\tau^*(H)$ is the minimum of 
$\sum_{v \in V} g(v)$ over all fractional covers $g$. Linear programming duality implies  $\nu(H)\le \nu^*(H)= \tau^*(H)\le \tau(H)$.   
We define
$$g^*(k,m)=\sup\{\tau^*(J) \mid J \in \ch(k,m), ~\nu(J)=1\}$$
$$h^*(k,m)=\sup \Big\{\frac{\tau^*(J)}{\nu(J)} \mid J \in \ch(k,m)\Big\},$$
and 
$$j^*(k,m)=\sup \Big\{\frac{\tau(J)}{\nu^*(J)} \mid J \in \ch(k,m)\Big\}.$$
Let   $\nu^{*(m)}(H)=\nu^*(H^{(m)})$ and $\tau^{*(m)}(H)=\tau^*(H^{(m)})$.

\subsection{A  general conjecture}

For a graph $G$ let $T(G)$ be the $3$-uniform hypergraph
whose vertex set is $V(G)$, and whose edges are all  triples of vertices  forming triangles in $G$. Tuza's conjecture can be written as:
\begin{conjecture} \cite{tuza} If $H=T(G)$ for a graph $G$, then  $\tu(H) \le 2\nt(H)$.
 \end{conjecture}
It may well be the case that this is true for all $3$-uniform hypergraphs. 

\begin{conjecture}\label{general32}
For every $3$-uniform hypergraph $H$, $\tu(H) \le 2\nt(H)$.
\end{conjecture}

 Since there are simple examples showing $h(3,2) \ge 2$, the conjecture says that  $h(3,2) = 2$. 
 
The $2$-skeleton of a hypergraph $H$ is the graph obtained by connecting by an edge every pair of vertices $\{u,v\}$ that is contained in some edge of $H$. A hypergraph is $2$-{\em connected} if its $2$-skeleton graph is $2$-connected.

Tuza's conjecture has only two known $2$-connected examples showing sharpness,
 $G=K_4$ and $G=K_5$. 
  Conjecture \ref{general32}, by contrast, has an infinite family of $2$-connected instances in which equality is attained.
\begin{example}\label{ex1}
 For $n$ even, let $H$ be  the hypergraph  whose edges are all subsets of   size $3$ of $[n]$ containing $1$. Then  $\nt(H)=\nu(K_{n-1})=\frac{n-2}{2}$ (an example of a  $2$-matching of this size is
 $\{(1,2i,2i+1) \mid 1\le i \le \frac{n-2}{2}\}$),
 while $\tu(H)=\tau(K_{n-1})=n-2$ (a minimal cover is $\{(1,i) \mid 1<i< n$\}).
 \end{example}

This example shows that the class of triangle hypergraphs $T(G)$ is a  proper subset of  $\chtt$, and that Conjecture \ref{general32} is a proper extension of Tuza's conjecture. Apart from this family, all examples of $2$-connected hypergraphs showing tightness that we know are small: $T(K_4)$ and $T(K_5)$,
$\binom{[5]}{3}$ minus one triple, and $\binom{[5]}{3}$ minus two triples that intersect in one vertex.  In the last two  examples $\nt=2$ and $\tu=4$.
 The hypergraph $\binom{[4]}{3}$ minus one triple is also an example of tightness, but it falls within the scope of Example \ref{ex1}.

Conjecture \ref{general32} can be generalized to hypergraphs of all uniformities. 

\begin{conjecture}\label{hequalsg}
 $h(k,m)=g(k,m)$  for all $m\le k$.
\end{conjecture}

Conjecture \ref{general32} would follow from this and the fact proved below (Proposition \ref{gkk-1}), that 
$g(3,2)=2$.

Example \ref{ex1} can also be  generalized. Let $m<k$. For any $(k-1)$-uniform hypergraph $L$ let $x$ be a vertex not belonging to $V(L)$, and let $\tilde{L}= \{\{x\} \cup e \mid \ell\in L\}$. Then $\tau^{(m)}(\tilde{L})=\tau^{(m-1)}(L)$, $\nu^{(m)}(\tilde{L})=\nu^{(m-1)}(L)$
 and 
 $\tau^{*(m)}(\tilde{L})=\tau^{*(m-1)}(L)$. This observation  implies:
 
 \begin{proposition}\label{inherit}
 $g(k,m) \ge g(k-1,m-1)$, $h(k,m) \ge h(k-1,m-1)$, $h^*(k,m) \ge h^*(k-1,m-1)$, and $j^*(k,m) \ge j^*(k-1,m-1)$. 
 \end{proposition}
 
 In particular  
 \begin{equation}\label{h42}
     h(k,2) \ge h(k-1,1)=k-1
 \end{equation}

 As we shall later see, this inequality is in some cases strict, e.g., $h(4,2)\ge 4$.

Conjecture \ref{general32} indicates that 
$\chtt$ hypergraphs behave better than general $3$-uniform hypergraphs. Indeed, $\chtt$ is a narrow subclass of the class of $3$-uniform hypergraphs. For example, it is easy to check that $\chtt$ hypergraphs are linear. It is not  clear whether there is a finite list of forbidden substructures that characterizes $\chtt$ hypergraphs. But it is possible that a milder condition than being in $\chtt$  suffices to guarantee $\tau(H) \le 2\nu(H)$.

\begin{definition}
A {\em tent} in a hypergraph $H$ is a set of four edges $\{e_i, ~1\le i \le 4\}$ satisfying:
\begin{enumerate}
\item $\bigcap_{i=1}^3 e_i \neq \emptyset$,
\item $|e_4 \cap e_i|=1$ for all $i \in \{1,2,3\}$, and
\item $e_4 \cap e_i \neq e_4 \cap e_j$ for $i\neq j$.
\end{enumerate}
\end{definition}
Note that if $H$ is linear then a tent is  a set of four pairwise intersecting edges  $e,f,g,h$, such that $e\cap f\cap g =\{x\}$ and $x\notin h$ (see Figure \ref{fig:tent}).

\begin{figure}
\includegraphics[width=6cm]{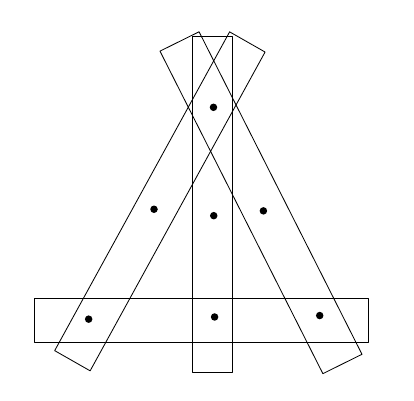}
\caption{A $3$-unfiorm linear tent.}
\label{fig:tent}
\end{figure}

\begin{lemma}\label{notent}
If $H$ is $k$-uniform then $H^{(k-1)}$ does not contain a tent.
\end{lemma}
\begin{proof}
Clearly, $H^{(k-1)}$ is linear. Assume for contradiction that it  contains a tent $e,f,g,h$ as above. 
Let $e',f',g',h'$ be edges in $H$ such  that   $e=\binom{e'}{k-1}, f=\binom{f'}{k-1}, g=\binom{g'}{k-1}, h=\binom{h'}{k-1}$. 
Let and $x=\{v_1,\dots,v_{k-1}\}$ be the intersection vertex of $e,f,g$. Then 
 $e'=\{v_1,\dots,v_{k-1},u_1\}, ~f'=\{v_1,\dots,v_{k-1},u_2\},\\  ~g'=\{v_1,\dots,v_{k-1},u_3\},$  where $v_i,u_j$ are all distinct vertices of $H$. Since $h'$ intersects each of $e,f,g$ at a vertex different than $x$, $h'$ must contain $u_1, u_2, u_3$ and  at least $k-2$ vertices from  $v_1,\dots,v_{k-1}$, implying $|h'|\ge k+1$, a contradiction.       
\end{proof}

   It is  easy to realize that also   $3$-partite $3$-uniform hypergraphs  cannot contain  tents. Hence it is tempting (though too daring to state as a conjecture) to pose the following problem, concerning a possible common generalization of Conjecture \ref{general32} and the $3$-case of Ryser's conjecture:

 \begin{problem}\label{daring}
  Is it true that a $3$-uniform hypergraph not containing a tent satisfies $\tau \le 2\nu$?
 \end{problem}
 
 A special family of hypergraphs that do not contain a tent is the family of hypergraphs $H$ with  $\Delta(H)\le 2$. For this family the answer is positive, a fact that is generalized in the following. 
 
 \begin{theorem}\label{delta=2}
 If  $H$ is an $k$-uniform hypergraph with $\Delta(H)\le 2$, then $\tau(H) \le \lceil\frac{k+1}{2}\rceil\nu(H)$.
 \end{theorem}
 The dual $D$ of $K_{k+1}$ is an example for sharpness of the theorem. The vertices of $D$ are the edges of $K_{k+1}$, and to each vertex $v$ of $K_{k+1}$ there corresponds the edge  
  $e_v=\{e \in E(K_{k+1}) \mid v \in e\}$ of $D$ (the star of $v$). Then $\Delta(H)=2$ because every vertex, namely  edge $uv$ 
 of $K_{k+1}$, 
 belongs to two stars, $e_u$ and $e_v$. Also,  $\tau(D)=\lceil\frac{k+1}{2}\rceil$ and $\nu(D)=1$.

 
 
We shall also see later that some of the results related to Conjecture \ref{general32} use only the  condition of not containing a tent (see Proposition \ref{genkriv} and Proposition \ref{gkk-1}). It will be of interest to deduce still further results from this condition.

\subsection{Aims, structure of the paper, and  main results}

Determining  precise values of $h(k,m)$, and even of $g(k,m)$, seems  difficult, as reflected by the fact that even the first non-trivial value, $h(3,2)$, is unknown.  Besides  generalizing Tuza's conjecture, the main aim of the paper is   obtaining bounds on these functions and on their fractional versions. 

In Section \ref{tuzabounds} we recapitulate some of the results known on Tuza's conjecture, whose proofs  go over to the same results on Conjecture \ref{general32}. 
In particular we mention, without proof, the following results whose restrictions to the  setting  of Tuza's conjecture are known: $$h(3,2) \le 3 - \frac{3}{23},$$ and $$j^*(3,2) = 2.$$ We also prove  that $h^*(k,k-1) \le k-1$, of which the case $k=3$ is proved in \cite{krivelevich}.  

In Section \ref{boundsongandgs} we determine   values of the above functions    for some parameters, and find bounds on these functions for other parameters. In particular we show: 
$$g^*(k,2)=\frac{k^2}{4}+o(k^2), $$  
$$g(k,k-1) = \Big\lceil{\frac{k+1}{2}}\Big\rceil, $$

$$ g(4,2) = 4,$$

$$2.5 \le h^*(4,2)\le 4.5,$$
and
$$j^*(4,2) \le 4.$$

If Conjecture \ref{hequalsg} is true, then the second result above will imply a generalization of Conjecture
\ref{general32} to all uniformities.

\begin{conjecture}\label{ceiltuza}
For every $k$-uniform hypergraph $H$, $\tau^{(k-1)}(H) \le \big\lceil{\frac{k+1}{2}}\big\rceil \nu^{(k-1)}(H)$. 
\end{conjecture}

In Section 5 we discuss the case of $k$-partite hypergraphs. If $H$ is $3$-partite, then so is $H^{(2)}$, and hence, by a result mentioned above, $\tuh/\nth \le 2$. We  show that if $H$ is a $3$-partite hypergraph then  $\tau^{*(2)}(H) / \nu^{(2)}(H) \le 9/5$, and if one of the vertex classes of $H$ is of size at most $2$ then $\tuh/\nth \le 5/3$. These results are probably far from being sharp: we do not know any example of a 3-partite hypergraph $H$ for which  $\tuh/\nth > 4/3$. We also prove  bounds on $\tau^{*(k-1)}(H) / \nu^{(k-1)}(H)$ for every $k$-partite hypergraph $H$.

\section
{Known bounds on Tuza's conjecture that go over to the setting of  Conjecture \ref{general32}}\label{tuzabounds}

The best bound known on the ratio 
$\tau^{(2)}(T(G))/\nt(T(G))$ for all graphs $G$
 was proved by Haxell \cite{haxell}. The proof applies, almost verbatim, to the general setting of Conjecture \ref{general32}. Namely:
\begin{theorem}\label{likehaxell}
For every $3$-uniform hypergraph $H$, $\tau^{(2)}(H) \le  \frac{66}{23} \nu^{(2)}(H)$. In other words, $h(3,2) \le \frac{66}{23}$.
\end{theorem}

Krivelevich \cite{krivelevich} proved two fractional versions of Tuza's conjecture.

\begin{theorem}\cite{krivelevich}\label{kriv}
The following are true for every graph $G$:
\begin{enumerate}
    \item $\tau^{(2)}(T(G)) < 2\nu^{*(2)}(T(G))$.
    \item $\tau^{*(2)}(T(G)) \le 2\nu^{(2)}(T(G))$. 
\end{enumerate}
\end{theorem}

Part (1) is true, with the same proof, for all $3$-uniform hypergraphs, namely:  
\begin{theorem}\label{genkriv} For every $3$-uniform hypergraph $H$, $\tau^{(2)}(H) < 2\nu^{*(2)}(H)$. 
\end{theorem}
This  cannot be improved in general, as is shown by the hypergraphs in Example \ref{ex1}. They satisfy
$\nst(H)=\nu^*(K_{n-1})=
 \frac{n-1}{2}$, and  thus the ratio $\frac{\tuh}{\nu^{*(2)}(H)} = \frac{2(n-2)}{n-1}$ approaches $2$ when $n \to \infty$. Combined with Theorem \ref{genkriv},  this entails
 $$j^*(3,2)=2.$$

The generalization of the second part of Theorem \ref{kriv} follows from Lemma \ref{notent}. 
\begin{proposition}\label{genkriv}
If $k\ge 3$, then $h^*(k,k-1) \le k-1$.
\end{proposition}
\begin{proof}
By a theorem of F\"uredi \cite{furedi},  a $k$-uniform hypergraph not containing a copy of a projective plane of uniformity $k$ satisfies $\tau^*\le (k-1)\nu$.   
Since for $k\ge 3$ a $k$-uniform projective plane contains a tent, the result follows by   Lemma \ref{notent}. 
  \end{proof}
  Note that Proposition \ref{genkriv} is not true for $k=2$, since in this case $\ch(k,k-1)$ is the set of all graphs.

\section{Some values and some bounds}\label{boundsongandgs}

\begin{proposition}\label{supattained}

The supremum in the definition of $g^*(k,m)$ is  attained.
\end{proposition}

\begin{proof}
 It suffices to show that for every number $t$ there
 are essentially finitely many hypergraphs in $\chkm$ with  $\tau^*\ge t$. Put rigorously, this means that there
 exists $p=p(k,m)$ such that if
 $H$ is a $k$-uniform hypergraph with more than $p$ edges and $\nu^{(m)}(H)=1$ then there exists a $k$-uniform hypergraph $K$ with $\nu^{(m)}(K)=1$, $|K|<|H|$, and $\tau^{*(m)}(K)\ge \tau^{*(m)}(H)$.

 By the  Sunflower Lemma \cite{sunflower} there exists $p$ such that if $|H|>p$
  then there exists a set of edges $S \subseteq E(H)$ having  $k+1$ edges,  such that all pairwise intersections of edges in $S$ are the same set $C$. Every  $f \in E(H)$ is disjoint from at least one of the (disjoint) sets $e \setminus C,~ e \in S$, and since the fact that $\nu^{(m)}(H)=1$ implies that $f$ intersects every edge in $H$ in at least $m$ vertices, we have
  $|f \cap C|\ge m$, and thus in particular $|C|\ge m$.

  Let $K$ be obtained from $H$ by
  replacing all  edges  of $S$ by the single edge $C \cup T$, where
  $T$ is a set of size $k-|C|$, disjoint from $V(H)$.
  To finish the proof it suffices to show that
  $\tau^{*(m)}(K)\ge \tau^{*(m)}(H)$, namely that for every fractional $m$-cover $g: \binom{V(K)}{m} \to \mathbb{R}^+ $ of $K$ there is a fractional $m$-cover
 of $H$ of equal or smaller size. For every set $a \in \binom{V(K)}{m}$ not contained in $V(H)$ (namely containing vertices from $T$) replace $g(a)$ by the same weight on $b$, where $b$ is a set of size $m$ containing $a \cap C$ and contained in $C$. Note that the same $b$ may aggregate weights from different sets $a$. The resulting function is a fractional $m$-cover in $H$, with  size at most $g$.
\end{proof}

\subsection{Bounds on $g(k,2)$ and on $g^*(k,m)$}
If $H$ is $k$-uniform and $\nth=1$, then for any edge $h \in H$ the set $\binom{h}{2}$ is a $2$-cover for $H$. Hence $g(k,2) \le {k\choose 2}$. This bound is not  attained for any  $k>2$. 
\begin{proposition}

 If $k>2$ then
 $g(k,2) < {k\choose 2}$.
\end{proposition}
\begin{proof}
Let $H$ be a $k$-uniform hypergraph, $k\ge 3$, with $\nu^{(2)}(H)=1$.
Write $$r = \max\{|e\cap f| \mid e,f \in E(H)\},$$ and let $t=\min\{r,\lceil k/2 \rceil\}$.  Then $t \ge 2$. We show that $\tau^{(2)}(H)\le {k\choose 2} - {t \choose 2}$.

Indeed, let $e,f$ be two edges in $H$ with $|e\cap f|\ge t$. Write
$$e = \{v_1,v_2,\dots,v_t, u_{t+1}, \dots ,u_k\},$$
$$f = \{v_1,v_2,\dots,v_t, w_{t+1}, \dots ,w_k\}.$$
Let $A_e$ and $A_f$ be the first $\lceil k/2 \rceil$ vertices of $e$ and $f$ in the above order, respectively. Thus $\{v_1,v_2,\dots,v_t\} \subset A_e \cap A_f$.
Let $B_e=e\setminus A_e$ and $B_f=f\setminus A_f$.
Then the set
$$\binom{A_e}{2} \cup \binom{A_f}{2} \cup (B_e \times B_f)$$
is a $2$-cover of $H$ of size at most ${k\choose 2} - {t\choose 2}$.
\end{proof}
Calculating $g(k,2)$ for general $k$ is probably hard, but  we know that it is of quadratic order.  The upper bound, $g(k)<\binom{k}{2}$, was  noted above. The lower bound
follows from the next theorem, and the fact that $g(k,m) \ge g^*(k,m)$.

\begin{theorem}\label{g*k2}
$g^*(k,2)= \frac{k^2}{4}+o(k^2)$.
\end{theorem}
\begin{proof}
 By the Prime Number Theorem, for $k$ large there exists a prime $p<\frac{k}{2}-1$ such that $\frac{k}{2}-p=o(k)$. Let $r=p+1$ and let $P$ be a projective plane of uniformity $r$. Let $K$ be the union of two disjoint copies $P_1, P_2$ of $P$, namely
$$K=\{e_1 \cup e_2 \mid e_1 \in E(P_1), ~e_2 \in E(P_2)\}.$$
Then $K$ is  $2r$-uniform and  $\nu^{(2)}(K)=1$. Every pair $a, b$ of distinct vertices of $K$ is contained in no more than
$r^2$ edges of $K$, so the constant function assigning value $\frac{1}{r^2}$  to every edge of $K$ is a fractional $2$-matching of $K$ of size $\frac{(r^2-r+1)^2}{r^2}$. This proves that $g^*(2r,2) \ge \frac{(r^2-r+1)^2}{r^2}$. This implies that $g^*(k,2) \ge \frac{(r^2-r+1)^2}{r^2}=\frac{k^2}{4}+ o(k^2).$

For the other direction, we prove that for every $k$, $g^*(k,2)\le \frac{k^2}{4} + k -2.$
Let $H$ be a $k$-uniform hypergraph with $\nu^{(2)}(H)=1$. If every two edges in $H$ meet at 3 or more vertices, then the function assigning $\frac{1}{3}$ to every pair of vertices contained in some fixed edge $e\in H$, and $0$ to every other pair of vertices, is a fractional $2$-cover of size $k(k-1)/6$.

Otherwise, there exist two edges $e,f \in H$ meeting at exactly two vertices. Write
$e=\{w_1,w_2,v_1,\dots,v_{k-2}\}$ and
$f=\{w_1,w_2,u_1,\dots,u_{k-2}\}$,
where the vertices $v_i$ and $u_j$ are all distinct. Since $\nth=1$, both $e$ and $f$ share at least two vertices with every edge in $H$.
This implies that the function $c$ defined below is a fractional $2$-cover for $H$:
\begin{equation*}
\begin{cases}
c(w_1w_2)=1\\
c(w_iv_j)=1, ~~1\le i\le 2, ~1\le j\le k-2\\
c(v_iu_j)=\frac{1}{4}, ~~1\le i\le k-2, ~1\le j\le k-2\\
\end{cases}
\end{equation*}
The size of $c$ is $\frac{(k-2)^2}{4} + 2k -3 = \frac{k^2}{4} + k -2,$ which yields the result.
\end{proof}

Next, we note that the example in the proof of Theorem \ref{g*k2} can be extended to $m>2$:

\begin{example}
Let $r$ be such that there exists a projective plane $P$ of
uniformity $r$. Let $P_1, \dots ,P_m$ be $m$ disjoint copies of $P$, and consider their join, $$H=P^{*m}
= \{e_1 \cup \cdots \cup e_m \mid e_i \in E(P_i), ~1\le i \le m \}.$$
Then $H$ is a $k$-uniform hypergraph for $k=mr$, and
 $\nu^{(m)}(H)=1$. Moreover, for $a_1,\dots,a_m \in V(H)$ we have
$\deg(a_1a_2\dots a_m) \le r^m$ in $H^{(m)}$, where the maximum is attained when $a_i\in
V(P_i)$, $1\le i \le m$. Hence the constant function
$f: E(H) \to \mathbb{R}^+$ assigning $\frac{1}{r^m}$ to every edge is a fractional $m$-matching, of size
$$\frac{(r^2-r+1)^m}{r^m} \ge  \frac{k^m}{m^m}- o(k^m).$$
\end{example}

Therefore, by the same argument as above, we have:

\begin{proposition}
For every $2\le m<k-1$, $g^*(k,m) \ge \frac{k^m}{m^m}- o(k^m).$
\end{proposition}

\subsection{The function $g(k,k-1)$.}
We wish to prove  that 
$g(k,k-1) = \lceil{\frac{k+1}{2}}\rceil$. This will follow from Theorem \ref{delta=2}, which we prove below. The proof  uses the following lemma.

 \begin{lemma}\label{brookslemma}
Let $H$ be a $k$ uniform hypergraph satisfyin $\Delta(H)\le 2$, and let $G=L(H)$. If $G$ is a complete graph or an odd cycle then  $\tau(H) \le \lceil \frac{k+1}{2}\rceil$.
\end{lemma}
\begin{proof}
Suppose first that $G$ is complete, meaning that  $\nu(H)=1$. Since $\Delta(H)\le 2$, every edge meets at most $k$ other edges, meaning that  $|E(H)|\le k+1$.  It follows that $\tau \le \lceil \frac{k+1}{2}\rceil$ (choose a vertex at a time belonging to  the intersection of two edges, and remove these edges). 

If $G$ is a cycle of length $2r+1$ for some $r\ge 1$, then $\nu(H)=r$ and $\tau(H)=r+1$. Assuming (as we may do) that $k>1$, it follows that   $\frac{\tau(H)}{\nu(H)} = \frac{r+1}{r} \le \lceil \frac{k+1}{2}\rceil$. 
\end{proof}

{\em Proof of Theorem \ref{delta=2}.}
It clearly suffices to prove the theorem for connected hypergraphs. 
Let $m=|E(H)|$, and let $G=L(H)$ be the line graph of $H$. Then $\nu(H)=\alpha(G)$. Let $M$ be an inclusion-wise maximal matching in $G$, and let $p=|M|$. By the maximality of $M$, the set $V(G) \setminus \bigcup M$ is independent, and thus $\alpha(G) \ge m-2p$, or 
$$p \ge \frac{m-\alpha(G)}{2}.$$ 
Each pair $\{e,f\}$ of edges of $H$ represented by
an edge in $M$ can be covered in $H$ by a vertex in $e \cap f$, and each edge of $H$  represented by a vertex in $V(G) \setminus \bigcup M$ can be covered by a single vertex. Hence $\tau(H) \le p+(m-2p)=m-p$. On the other hand, the fact that $\Delta(H)\le 2$ implies that $\Delta(G) \le r$. 

By the lemma, we may assume that $G$ is neither a complete graph nor an odd cycle, and since we are assuming that $G$ is connected,   by Brooks' theorem \cite{brooks},  $\alpha(G) \ge \frac{m}{r}$. Summarizing, we have:
$$\tau(H) \le m-p \le \frac{m+\alpha(G)}{2}\le
\frac{\alpha(G)(r+1)}{2}=\frac{\nu(H)(r+1)}{2}
,$$
proving the theorem. 
\qed

\medskip

\begin{proposition}\label{gkk-1}
$g(k,k-1) = \lceil{\frac{k+1}{2}}\rceil$.
\end{proposition}
\begin{proof}
Let show first that $g(k,k-1) \le \lceil{\frac{k+1}{2}}\rceil$, namely $\tau(H^{(k-1)}) \le \lceil{\frac{k+1}{2}}\rceil\nu(H^{(k-1)})$ for every $k$-uniform hypergraph $H$  with $\nu(H^{(k-1)})=1$. 
If $\Delta(H)\le 2$, then the inequality follows from Theorem   \ref{delta=2}.
So, we may assume that there exists $v \in V(H^{(k-1)})$ having degree greater than 2. By Lemma \ref{notent} $H^{(k-1)}$ does not contain a tent, which, since all edges in $H^{(k-1)}$ intersect, means that there is no edge missing $v$,  implying  $\tau(H^{(k-1)})=1$.

For the converse  inequality  consider the hypergraph $H={[k+1]\choose k}$. Clearly, $\nu^{(k-1)}(H)=1$. For a $(k-1)$-set $t \subset [k+1]$, the complement $t^c=[k+1]\setminus t$ is a pair, and a set of $(k-1)$-tuples $t_i$ is a $(k-1)$-cover for $H$ if the corresponding pairs $t_i^c$ cover all vertices of $[k+1]$. Hence $\tau^{(k-1)}(H) =  \lceil{\frac{k+1}{2}}\rceil$.
\end{proof}



\subsection{Some values of the functions, and some bounds}
 
\begin{proposition}\label{g42}
$g(4,2) = 4$.
\end{proposition}
\begin{proof} 
Let $H$ be the hypergraph on vertex set $\{a,b,c,d,e,f,g\}$ with edge set $\{abcd,~ abef, ~cdef, ~aceg, ~bdeg, ~adfg, ~bcfg\}$. Then $\nt(H)=1$ and $\tu(H)=4$. This shows  $g(4,2) \ge 4$.

To prove the converse inequality, let $H$ be a $4$-uniform hypergraph with $\nth=1$.
We show that $\tuh \le 4$. Let $e$ be an edge in $H$. Call a pair $a \in \binom{e}{2}$ {\em indispensable} if there is an edge $f \in E(H)$ such that $f \cap e=a$, and otherwise call it  {\em dispensable}.
Note that if an edge $f\neq e$ contains two pairs from $\binom{e}{2}$, then it contains also a third pair. Hence,
if $a,b \in \binom{e}{2}$ are dispensable pairs, then $\binom{e}{2} \setminus \{a,b\}$ is a 2-cover for $H$ of size 4. Hence we may assume that there exists at most one dispensable pair. So, there exist indispensable pairs $a,a',b,b'$ such that $a \cap a' =b \cap b'=\emptyset$.

Let $f,f',g,g'$ be edges in $H$ witnessing the indispensability of the pairs $a,a',b,b'$, respectively, namely,  $f \cap e=a,~f'\cap e = a',~ g\cap e=b, ~g'\cap e=b'$. Since $\nth=1$, $f \cap f'$ is a pair $x$ disjoint from $e$. There is no other edge intersecting $e$ just at $a$, since such an edge would not intersect $f'$ in two vertices or more. Similarly, also $f',g,g'$ are the only witnesses to the indispensability of $a',b,b'$ respectively. Let $y=g\cap g'$. By the above argument, the set of pairs $\binom{e}{2}\cup \{x,y\}\setminus \{a,a',b',b'\}$ is a 2-cover for $H$ of size $4$, proving the theorem.
\end{proof}

\begin{proposition}
$2.5 \le h^*(4,2) \le 4.5$.
\end{proposition}

\begin{proof}
The lower bound is valid since $g^*(4,2) \ge 2.5$, as shown by the hypergraph $H=\binom{[6]}{4}$, for which $\nth=1$ and $\tst(H)=2.5$ (an optimal fractional $2$-matching gives each of the $15$ edges weight $1/6$,  and an optimal fractional $2$-cover gives each of the $15$ pairs of vertices  weight $1/6$).

To prove the upper bound, let $H$ be a 4-uniform hypergraph, and let $\nth = \nu$.  We shall show that $\tau^{*(2)}(H) \le 4.5\nu$. We construct a fractional 2-cover of $H$ as follows.  Let $M=\{m_1,\dots,m_{\nu}\}$ be a $2$-matching in $H$. First,  assign weight  $1/2$ to every pair in  $P(M)=\bigcup_{i=1}^{\nu} \binom{m_i}{2}$.
This amounts to a total weight of $3k$, and every edge in $H$ that intersects each edge in $M$ in at least two pairs
is fractionally $2$-covered.

For every $1\le i \le \nu$ and every pair $p \in \binom{m_i}{2}$, let $q(p)=m_i \setminus q$, and let $H(p)$ be the set of edges $e \in  H$ for which $e \cap m_i=p$, and $|e \cap m_j|<2$ for all $i\neq j$. Note that $\bigcup_{i=1}^{\nu}\bigcup_{p \in \binom{m_i}{2}} H(p)$ is the set of those edges that are not yet fractionally $2$-covered.

If for some $p \in \binom{m_i}{2}$ we have $e\in H(p)$ and $f\in H(q(p))$,
then $|e\cap f \setminus m_i|=2$, since otherwise $M \cup \{e,f\} \setminus \{m_i\}$ is a $2$-matching of size $\nu+1$. This means that there exists a pair $r(p)$ of vertices contained in all edges belonging to $H(p) \cup H(q(p))$. In particular, if  $H(p)=\emptyset$ then $r(p)=q(p)$.
We assign now additional weights of $\frac{1}{2}$ to $r(p)$, for every $p \in P(M)$ (so, if $r(p)=q(p)$ then it now has  weight $1$).
We have thus added weight $\frac{1}{2}$ for every one of the three pairs of disjoint vertex couples in $\binom{m_i}{2}$, so in total we added $\frac{3k}{2}$ to the fractional 2-cover we construct. Now every edge in $H$ is fractionally $2$-covered, and the total assigned weight is at most $4.5\nu$.
\end{proof}

\begin{proposition}
$j^*(4,2)\le  4$.
\end{proposition}
\begin{proof}
We shall show that
if $H$ is a $4$-uniform hypergraph then $\tuh < 4\tau^{*(2)}(H)$.
Assume, for contradiction, that there exists a  4-uniform hypergraph $H$ for which $\tuh \ge 4\tst(H)$, and let $H$ be minimal with respect to this property. 

Let $g,f$ be a minimal fractional 2-cover and maximal fractional 2-matching in $H$, respectively.
Let $U$ be the set of all pairs $u\in \binom{V(H)}{2}$ for which $g(u)>0$. By a complementary slackness condition,
\begin{equation}\label{sizeu}
|U| = \sum_{u\in U} \sum_{e\in \hu: u\in e} f(e) \le 6\nst(H).\end{equation}

If there exists $u\in U$ with $g(u)\ge 1/4$, let  $H'$ be the hypergraph obtained from $H$ by removing all edges in $H$ containing $u$. Then $\tuh \le \tu(H')  + 1$ and $\tst(H)\ge \tst(H')+\frac{1}{4}$. By the minimality assumption on $H$ we have $\tu(H') < 4 \tst(H')$, implying that
\begin{equation*}
\begin{split}
\tuh &\le \tu(H')  + 1 < 4\tst(H')+1  \\& \le 4(\tst(H) - 1/4) + 1 = 4\tst(H),
\end{split}
\end{equation*}
contradicting the assumption on $H$.

Thus we may assume that $g(u) < 1/4$ for all $u\in U$, implying  that every edge $h\in H$ contains at least $5$ members of $U$. Considering $U$ as a graph on $V(H)$, there is a partition of $V(H)$ into sets $A, B$
such that at least $|U|/2$ edges in $U$ are $(A,B)$-crossing, namely, have a non-empty intersection with both $A$ and $B$. The set $$\Big(U \cap \binom{A}{2}\Big)\cup \Big(U \cap \binom{B}{2}\Big)$$
forms a $2$-cover for $H$, since in every edge in $H$ there can be only $4$ $(A,B)$-crossing edges from $U$, and as noted above every edge in $H$ contains $5$ pairs belonging to $U$.
Thus we have $\tuh \le |U|/2$, which, together with \eqref{sizeu} yields
$$\tuh \le |U|/2 \le 3\tst(H) < 4\tst(H),$$
again contradicting the assumption on $H$.
\end{proof}

\section{$k$-partite hypergraphs}

A $k$-uniform hypergraph $H=(V,E)$ is called {\em $k$-partite}
if $V$ is the disjoint union of sets $V_1, \ldots ,V_k$ (called {\em vertex classes}), and every edge in $E$ intersects every $V_i$ at one vertex. If $H$ is $k$-partite then $\hm$ is $\binom{k}{m}$-partite, with  vertex classes
indexed by the sets $A \in \binom{[k]}{m}$. 
For each $e \in E(H)$ the edge $\binom{e}{m}$ of $\hm$ consists of one $m$-tuple in each class indexed by a set $A$, namely $e \cap \bigcup_{i \in A}V_i$.

$k$-partite hypergraphs behave particularly well with respect to matchings. The best-known result in this direction is K\"onig's theorem, stating that in bipartite graphs $\tau=\nu$.
Lov\'asz \cite{lov} proved that in $k$-partite hypergraphs $\tau< \frac{k}{2}\nu^*$, and as already noted, by a theorem of  F\"uredi  \cite{furedi}  in such hypergraphs $\tau^*\le (k-1)\nu$.

As mentioned above, the $k=3$ case of Ryser's conjecture, namely $\tau\le 2\nu$ in $3$-partite hypergraphs, is known. If $H$ is $3$-partite then so is $\hu$, and thus $\tu(H)\le 2\nt(H)$. In \cite{haxellkohaya} this was strengthened to:
\begin{theorem}\label{epspartite}
If $G$ is a $3$-partite graph then
$\tu(T(G)) < 1.956\nt(T(G))$.
\end{theorem}
Careful scrutiny shows that the proof also yields the more general 

\begin{theorem}\label{epspartite}
If $H$ is a $3$-partite hypergraph then
$\tu(H) < 1.956\nt(H)$.
\end{theorem}

In the following example, $H$ is a $3$-partite hypergraph with
$\nt(H)=3$ and $\tu(H)=4$.  
\begin{example}\label{43}
Let  $V=\{a_1,a_2\}\cup\{b_1,b_2,b_3\}\cup\{c_1,c_2\}$ and let $E$ consist of the edges
$~a_1b_1c_1,
~a_1b_1c_2,
~a_1b_2c_2,
~a_2b_2c_2,
~a_2b_2c_1,
~a_2b_3c_1,
~a_1b_3c_1$.
\end{example}
We do not know any $3$-partite hypergraph $H$ in which $\frac{\tu(H)}{\nt(H)}>\frac{4}{3}$. It is easy to see that if $H$ is $3$-partite and $\nt(H)=1$ then $\tu(H)=1$, so this example shows that the analogue of Conjecture \ref{hequalsg} is false in the $3$-partite case.

We shall be able to prove some bounds smaller than $2$ on the ratio $\frac{\tu(H)}{\nt(H)}$ for a  $3$-partite hypergraph $H$ under two special conditions on $H$: having one vertex class of size $2$, or having identical neighborhoods for all vertices in one vertex class.

Let $H$ be a $3$-partite graph, with vertex classes $A,B,C$. The
vertex classes of $\hu$ are $AB, BC$ and $AC$, where  $XY=\{xy \mid x \in X, y\in Y\}$ for any two sets $X,Y$.
We introduce an asymmetry, by singling out one vertex class of $H$, say $A$. Write $A=\{a_1, \ldots ,a_p\}$, and for each $i \le p$ let $F_i=\{bc \mid a_ibc \in H\}$ and $\cf=(F_1, \ldots,F_p)$.
Then $F_i$ are sets of edges in a bipartite graph. Given such a family $\cf$, let $\nt(\cf)$ be the maximum of $|\bigcup_{i \le p}N_i|$, where each $N_i$ is a matching in $F_i$. Let also $\tu(\cf)=\min_{Z \subseteq B \times C} |Z|+\sum_{i \le p} \tau(F_i-Z)$.

\begin{claim}\label{equivalentnutau}
\hfill
\begin{enumerate}
\item
$\nt(\cf)=\nth$.
\item
 $\tu(\cf)=\tu(H)$.
\end{enumerate}
\end{claim}
\begin{proof}
(1)~~Given a $2$-matching $M$ in $H$, let $N_i=\{bc \mid a_ibc \in M\}$. Then each $N_i$ is a matching in $F_i$, and $M=\bigcup_{i \le p} N_i$.  Conversely, if $N_i,~~i \le p$ are  matchings in the respective sets $F_i$ then $\bigcup_{i \le p} \{a_ibc \mid bc \in N_i\}$ is a $2$-matching in $H$ of size $|\bigcup N_i|$.

(2)~~Given a $2$-cover $Q$ of $H$, let $Z=Q \cap (B \times C)$, and let $T_i=\{x \in B \cup C \mid a_ix \in Q\}$. Then, for $i \le p$,  $T_i$ is a cover for $F_i -Z$, 
and $|Q|= |Z|+\sum_{i\le p}|T_i|$. This proves that  $\tuh \le \tu(\cf)$. The converse inequality is proved in a similar way.
\end{proof}

By Theorem \ref{epspartite} we have $\tu(\cf) < 2\nt(\cf)$. If $|\cf|=1$ then by K\"onig's duality theorem $\tu(\cf)=\nt(\cf)$, meaning that if $|A|=1$ then $\tuh=\nth$. Another case of equality is given in the following:
\begin{proposition}\label{samef}
If $F_i=F_j$ for all $1\le i <j\le p$ then $\nt(\cf)=\tu(\cf)$.
\end{proposition}

Recall that a set $K$ of edges in a graph is a {\em $p$-factor} if $\Delta(K):=\max_{v \in V}\deg_K(v) \le p$. Let $\nu_p(G)$ be the maximal size of a $p$-factor in a graph $G$. By K\"onig's edge coloring theorem, if $G$ is bipartite then $\Delta(K)\le p$  if and only if  $K$ is the union of $p$ matchings. Letting $G=F_i$, we see that the following is a re-formulation of Proposition \ref{samef}.
\begin{proposition}\label{samesame}
If $G$ is bipartite then $\nu_p(G) =\min_{Z \subseteq E(G)}|Z|+p\tau(G-Z)$.
\end{proposition}
\begin{proof} To show that $\nu_p(G) \le |Z|+p\tau(G-Z)$ for every subset $Z$ of $E(G)$, let $C$ be a cover for $G-Z$, and consider a $p$-factor  $F$. Since every vertex in $C$ is incident with at most $p$ edges in $F$, We have 
$$|F|=|F \cap Z|+|F \setminus Z|\le |Z|+p|C|.$$

To prove the converse inequality, suppose that the respective vertex classes of $G$ are $B, C$, and let $\cm$ be the matroid on $E(G)$ consisting of those sets of edges $F$ such that $\deg_F(b) \le p$ for every $b  \in B$, and let $\cn$ be the matroid consisting of those sets of edges $F$ such that $\deg_F(c) \le p$ for every $c \in C$. 
Note that for a set $K$ of edges 
\begin{equation}\label{rank} rank_\cm(K)=\sum_{b \in B}\min(\deg_K(b),p),\end{equation}
and similarly for $\cn$.

A set of edges is a $p$-factor if and only if it belongs to $\cm \cap \cn$, and hence, by  Edmonds' two matroids intersection theorem,     $\nu_p(G)$ is the minimum, over all partitions $(E_1, E_2)$ of $E(G)$, of $rank_\cm(E_1)+rank_\cn(E_2)$.

Let $(E_1, E_2)$ be a partition in which this minimum is attained.  Set 
$T_1=\{bc \in E_1 \mid \deg_{E_1}(b)\ge p\}$ and $T_2=\{bc \in E_2 \mid \deg_{E_2}(c)\ge p\},$ and let $Z=E(G) \setminus (T_1 \cup T_2)$. Then the set 
$$\{b \in B \mid \deg_K(b) \ge p\}\cup \{c \in C\mid \deg_K(c) \ge p\}$$ is a cover for $G-Z$, and hence, by \eqref{rank}, we have
$$rank_\cm(E_1)+rank_\cn(E_2) \ge |Z|+k\tau(G-Z),$$
proving  that $\nu_p(G) \ge |Z|+p\tau(G-Z)$. 
\end{proof}

\begin{corollary}
If the neighborhoods of all vertices in one vertex class of a $3$-partite hypergraph $H$ are identical, then $\nt(H)=\tu(H)$.
\end{corollary}

Another  case in which we can improve the upper bound on the ratio $\tu/\nt$ is that of  $|\cf|=2$. 
\begin{theorem}\label{tunu}
For any two sets $F_1,F_2$ of edges in a bipartite graph
$$\tu(F_1,F_2) \le \frac{5}{3}\nt(F_1,F_2).$$
\end{theorem}
\begin{proof}
Let $N$ be a maximum matching in $F_1 \cap F_2$, and let $n=|N|$. 
Let $L_i$ be a maximum matching in $F_i \setminus N$, and let $\ell_i=|L_i| ~~(i=1,2)$.
By the maximality of $N$ we have $|L_1 \cap L_2| \le n$, 
implying $$|L_1 \cup L_2| \ge \ell_1+\ell_2-n.$$

Assume without loss of generality  that $\ell_1\ge \ell_2$.
By the definition of $\nt(F)$ we have
$$\nt(F) \ge \max(n+\ell_1, |L_1 \cup L_2|)\ge \max(n+\ell_1,\ell_1 +\ell_2-n).$$
On the other hand, taking $Z=N$ in  the definition of $\tu(F)$  yields
$$\tu(F) \le n +\ell_1+\ell_2.$$

Write  $\ell_i=(1+\alpha_i)n$ (where $\alpha_i$ may be negative).
By the above, in order to prove the theorem it suffices to show  that $$3+\alpha_1+\alpha_2\le
\frac{5}{3}\max(2+\alpha_1, \alpha_1+\alpha_2+1).$$ 

If $\alpha_2\le 1$ then the inequality is $3+\alpha_1+\alpha_2 \le \frac{5}{3}(2+\alpha_1)$, 
which after cancellations becomes $\alpha_2 \le \frac{2}{3}\alpha_1+\frac{1}{3}$, which is true since $\alpha_2\le \alpha_1$ and $\alpha_2\le 1$. 
If $\alpha_2 \ge 1$ the inequality is $3+\alpha_1+\alpha_2\le \frac{5}{3}(1+\alpha_1+\alpha_2)$. This is valid since $\alpha_2\le \alpha_1$, so $\alpha_1+\alpha_2 \ge 2$.
\end{proof}

Combined with Claim \ref{equivalentnutau} this yields:
\begin{corollary}
If $H$ is a $3$-partite hypergraph with a vertex class of size $2$,  then $\tau^{(2)}(H) \le \frac{5}{3} \nu^{(2)}(H)$.
\end{corollary}

We conclude with a bound on the ratio between the fractional 
  $(k-1)$-covering number and the integral $(k-1)$-matching number in $k$-partite hypergraphs.  
\begin{theorem}\label{tau*k-1}
If $H$ is a $k$-partite hypergraph then $$\tau^{*(k-1)}(H)\le \frac{k^2}{2k-1}\nu^{(k-1)}(H).$$ In particular, a $3$-partite hypergraph $H$ has $\tau^{*(2)}(H)\le 1.8\nu^{(2)}(H)$.
\end{theorem}

To prove the theorem, let $H$ be a $k$-partite hypergraph with $n=\nu^{(k-1)}(H)$, and let $M$ be a  $(k-1)$-matching  in $H$ of size $n$.
An edge $e \in E(H) \setminus M$ is said to {\em mimic} an edge $m \in M$ if $(M\setminus \{m\})\cup\{e\}$ is also a  $(k-1)$-matching (of size $n$) in $H$. By the maximality of $M$, $M\cup\{e\}$ is not a $(k-1)$-matching, and hence $e$ mimics $m$ if and only if
$|m\cap e|=k-1$, and for every edge $m'\in M$, $m'\ne m$, we have $|m' \cap e|< k-1$.
Let $M_1 = \{m_1,\dots,m_t\}$ be the set of those edges in $M$ that have a mimicking edge, and let $M_2 = \{m_{t+1},\dots,m_{n}\}$  be the set of those edges that do not have a mimicking edge. For every $1\le i \le t$ let $F_i \subseteq E(H)\setminus M$ be the set of all edges mimicking $m_i$. If $F$ is a hypergraph, we denote by $\bigcap F$ the intersection of its edges.

\begin{claim}\label{friendsn}
For every $1\le i \le t$, $|\bigcap F_i \cap m_i| = k-1$.
\end{claim}

\begin{proof}
Suppose that this is false. Then there exist edges $f,g \in F_i$ such that $f\cap m_i \neq g\cap m_i$.  
Since $|f\cap m_i|= |g\cap m_i|=k-1$, and $H$ is $k$-partite, this implies that  
 $|f\cap g|\le k-2$. Therefore, $M \cup \{f,g\} \setminus \{m_i\}$ is a $(k-1)$-matching of size $n+1$ in $H$, a contradiction.
\end{proof}
For every $1\le i \le t$ let $p_i = \bigcap F_i \cap m_i$. 
\begin{claim}\label{g1n}
The function $g_1:{V(H)\choose k-1} \rightarrow \mathbb{R^+}$, in which every $p_i$, $1\le i \le t$, is mapped to 1, and every other set of $k-1$ vertices contained in $m_i$, $1\le i\le n$, is mapped to $1/2$, is a fractional $(k-1)$-cover of $H$. Its size is $\frac{kn}{2} + \frac{t}{2}$.
\end{claim}

\begin{proof}
Every edge $h\in H$ intersects an edge in $M$ in at least $k-1$ vertices. Clearly, if $h \in M$ then  it is fractionally covered by $g_1$. If $h$ is a mimicking edge for some $m_i \in M$, then by Claim \ref{friendsn} it contains $p_i$, and therefore it is fractionally covered by $g_1$. Finally, if $h \notin M$ is not a mimicking edge, then it must intersect at least two edges in $M$ in $k-1$ vertices each, and therefore it is fractionally covered by $g_1$. 
\end{proof}

Another useful observation is the following. 
\begin{claim}\label{pro1}
Let $e$ be an edge in $E(H)\setminus M$, such that for every $1\le i\le t$, $p_i$ is not contained in $e$. Then there exists $m \in M_2$ such that $|e \cap m|= k-1$.
\end{claim}
\begin{proof}
Assume that there exists $e \in E(H)\setminus M$ contradicting the proposition. Then by Claim \ref{friendsn}, $e$ lies in $E(H) \setminus \Big(M \cup \bigcup_{i=1}^t F_i \Big)$. Therefore, $e$ intersects at least two edges in $M$ in $k-1$ vertices each. If $e$ does not intersect an edge in $M_2$ in $k-1$ vertices, then it
must intersect at least two edges in $M_1$, each one in $k-1$ vertices. Let $\{m_j \mid j\in J\}$, $J \subseteq [t]$, be those edges in $M_1$ intersecting $e$ in $k-1$ vertices. For every $j \in J$ choose $f_j \in F_j$. 
We claim that the set $$M \cup \{f_j\mid j\in J\} \cup \{e\}\setminus \{m_j \mid j\in J\}$$ is a $(k-1)$-matching in $H$ of size $n+1$, which is clearly a contradiction.  Claim \ref{pro1} will follow from  the next two claims:
\begin{claim}
For every $j\in J$, $|e \cap f_j| \le k-2$.
\end{claim}

Indeed, both $e$ and $f_j$ intersect $m_j$ in $(k-1)$ vertices, and $e\cap m_j \neq f_j\cap m_j$ (since $f_j \cap m_j = p_j$ and $e$ does not contain $p_j$). Since $H$ is $k$-partite, this implies that $|e \cap f_j| \le k-2$.

\begin{claim} \label{last}
For every $i,j\in J$, $|f_i \cap f_j| \le k-2$.
\end{claim}

 suppose to the contrary that  $|f_i \cap f_j| = k-1$. Since $H$ is $k$-partite and $|m_i\cap m_j| \le k-2$, $|m_i\cap f_j| \le k-2$, and $|m_j\cap f_i| \le k-2$, we must have that
$$m_i = \{v_1,\dots, v_{k-2}, u_1,u_2\},~
m_j = \{v_1,\dots, v_{k-2}, w_1,w_2\},$$
$$f_i = \{v_1,\dots, v_{k-2}, u_1, z \},~ 
f_j = \{v_1,\dots, v_{k-2}, w_1, z \},$$
 where $v_i,u_i, w_i, z$ are all distinct vertices.
Moreover, we have $p_i = \{v_1,\dots, v_{k-2}, u_1\}$ and $p_j = \{v_1,\dots, v_{k-2}, w_1\}$.
But now, the fact that $e$ intersects both $m_i$ and $m_j$ in $k-1$ vertices each implies that either $e= \{v_1,\dots, v_{k-2}, u_1,w_2\}$ or $e= \{v_1,\dots, v_{k-2}, w_1,u_2\}$, and thus  it contains either $p_i$ or $p_j$, contradicting the assumption on $e$.

This concludes the proof of Claim  \ref{last} and thus also of Claim \ref{pro1}
\end{proof}

Claim \ref{pro1} implies:
\begin{claim}
\label{g2n}
The function $g_2:{V(H)\choose k-1} \rightarrow \mathbb{R^+}$, in which every $p_i$, $1\le i \le t$, and every set in $\binom{m_j}{k-1}$, $t+1\le j \le n$, are mapped to $1$, is a $(k-1)$-cover of $H$. Its size is $kn-(k-1)t$.
\end{claim}

We are now ready to prove Theorem \ref{tau*k-1}. 
If $t \le \frac{kn}{2k-1}$ then by Claim \ref{g1n} there exists a fractional $(k-1)$-cover of $H$ of size $\frac{kn}{2} + \frac{t}{2} \le \frac{k^2n}{2k-1}.$
If $t > \frac{kn}{2k-1}$, then by Claim \ref{g2n} there exists a $(k-1)$-cover of $H$ of size $kn-(k-1)t < \frac{k^2n}{2k-1}$.


\end{document}